\documentclass[conference]{IEEEtran}
\IEEEoverridecommandlockouts
\usepackage{graphicx} 

\usepackage{graphicx}%
\usepackage{multirow}%
\usepackage{amsmath,amssymb,amsfonts}%
\usepackage{amsthm}%
\usepackage{mathrsfs}%
\usepackage{mathabx}
\usepackage[dvipsnames]{xcolor}%
\usepackage{textcomp}%
\usepackage{manyfoot}%
\usepackage{booktabs}%
\usepackage{algorithm}%
\usepackage{algorithmicx}%
\usepackage{algpseudocode}%
\usepackage{listings}%
\usepackage{enumerate}
\usepackage{ulem}
\usepackage{comment}

\usepackage{tabularx}

\usepackage{stmaryrd} 
\usepackage{dsfont}   
\usepackage{yfonts}  

\newcommand {\R} {\mathbb{R}}

\renewcommand{\L}{\mathcal{L}}

\newcommand{\eps}{\varepsilon}

\newcommand{\vphi}{\varphi}
\newcommand{\sigmin}{\sigma_\text{min}}

\newcommand{\pierre}[1]{\textcolor{black}{#1}}
\newcommand{\bruno}[1]{\textcolor{black}{#1}}

\newtheorem{theorem}{Theorem}
\newtheorem{prop}{Proposition}%
\newtheorem{lemma}{Lemma}%
\newtheorem{corollary}{Corollary}%
\newtheorem{remark}{Remark}%

\def\BibTeX{{\rm B\kern-.05em{\sc i\kern-.025em b}\kern-.08em
    T\kern-.1667em\lower.7ex\hbox{E}\kern-.125emX}}

\begin{document}

\title{Stability estimates for adaptive focused time-frequency transforms\\
\thanks{This work was supported by the French National Agency for Research through the BMWs project (ANR-20-CE45-0018)}
}

\author{\IEEEauthorblockN{Pierre Warion}
\IEEEauthorblockA{\textit{Aix-Marseille Univ., CNRS,} 
\textit{I2M, Marseille, France} \\
pierre.warion@univ-amu.fr}
\and
\IEEEauthorblockN{Bruno Torr\'esani}
\IEEEauthorblockA{\textit{Aix-Marseille Univ., CNRS,} 
\textit{I2M, Marseille, France} \\
bruno.torresani@univ-amu.fr}
}

\maketitle

\begin{abstract}
This contribution is a follow-up of a recent paper by the authors on adaptive, non-linear time-frequency transforms, focusing on the STFT based transforms. The adaptivity is provided by a focus function, that depends on the analyzed function or signal, and that adapts dynamically the time-frequency resolution of the analysis. Sticking to the continuous case setting, this work provides new stability results for the transform (stability with respect with the focus function). It also investigates in some details focus functions based upon regularized Rényi entropies and show corresponding continuity results.
\end{abstract}

\begin{IEEEkeywords}
Adaptive time-frequency transforms, time and frequency focus, entropy-based focus functions.
\end{IEEEkeywords}


\section{Motivations and introduction}



Time-frequency transformations play an important role at the interface of mathematical analysis, signal processing and many application fields where non-stationarity needs to be taken into account. A prototype for time-frequency transformation is the short-time Fourier transform (STFT), one of whose features is constant time-frequency resolution. In~\cite{warion2024class}, an adaptive variant $f\in L^2(\R)\to Mf\in L^2(\R^2)$ of STFT was proposed and studied, in which the time-frequency resolution can be dynamically adapted to signal variations. Adaptation is achieved by composing the analysis window with a focus function $\sigma_f$, dependent on the signal being analyzed $f$. Adaptation thus results in a non-linear time-frequency transform.

The main results of~\cite{warion2024class} concern the definition of such adaptive transforms, estimates of transform norms, and another variant based on wavelet transforms, which will not be addressed here. Ref~\cite{warion2024class} also provides numerical illustrations obtained with focus functions based on entropy measures. These were provided for the sake of completeness, to show the effect of adaptivity, without being analyzed in-depth. By inversion scheme, we mean the following problem: suppose one is given the non-linear transform $Mf$ of $f\in L^2(\R)$, computed with the $f$-dependent focus function $\sigma_f$ (the mapping $f\to\sigma_f$ beong known), reconstruct $f$ without knowing $\sigma_f$.

In this contribution, we complement this work with stability results for the adaptive transform, explicitly taking into account its non-linearity. We also take a closer look at entropy-based focus functions, for which we propose a regularized version, show that it satisfies the assumptions made in the definition of focused transforms. We also derive bounds for such focus functions, and show continuity results. These results should represent a step towards the construction of an inversion scheme for the transform, which is an open problem.

\paragraph*{Outline}
This paper is organized as follows. We first recall in Section~\ref{sec:tftft} the definition and main properties of the time-focused STFT, and its inverse when the focus function is fixed (independent of the analyzed signal). We then derive in Section~\ref{sec:continuity.focus} continuity results for the non-linear transform and a corresponding left-inverse. We also introduce and discuss regularized entropy-based focus functions.

Proofs of the main results are only briefly sketched, more complete proofs will be published in a forthcoming paper, together with more detailed analyses.

\paragraph*{Notations}
Throughout this paper, $\langle\cdot,\cdot\rangle$ denotes the $L^2(\R)$ inner product, $L^p$ norms for univariate functions are denoted by $\|\cdot\|_p$, and the $L^2(\R^2)$ norm is denoted by $\|\cdot\|_{L^2(\R^2)}$.

\section{Time-focused transform}
\label{sec:tftft}
We first recall the main aspects of the non-linear time-focused transform introduced in~\cite{warion2024class}. 
\subsection{Definition and properties}

Suppose there is a mapping associating to every $f\in L^2(\R)$ a \textit{focus function}, $\sigma_f$, assumed to be continuous, greater than \bruno{ some constant $1>\sigmin>0$} and to tend to $1$ at $\pm\infty$ (the class of such functions is denoted by \bruno{$\Phi_{\sigmin}$}).  Let $h$ be a differentiable function, that satisfies the decay condition:
$\exists \vphi_h,\psi_h$ continuous, non negative and bounded such that
\begin{equation}
\label{fo:window.assumptions}
    \forall t\in\R, \ |h(t)|\leqslant \frac{\psi_h(t)}{1+|t|} \quad \text{and} \quad |h'(t)|\leqslant \frac{\vphi_h(t)}{1+t^2}.
\end{equation}
Given the above hypotheses, we can define the  time-focused atoms for $\forall x,t,\omega \in\R$,
\begin{equation}\label{fo:time.focused.atoms}
	 h_{t,\omega,\sigma_f}(x):=  \sqrt{\sigma_f(t)} e^{2i\pi\omega x}h(\sigma_f(t)(x-t))\ ,
\end{equation}
and the corresponding transform of a signal $f\in L^2(\R)$ by
\begin{equation}\label{eq_def_operator_time_focus}
	\forall t,\omega\in\R, \ M f(t,\omega) := \langle f , h_{t,\omega,\sigma_f} \rangle\ .
\end{equation}
This transform has similarities with continuous warped transforms studied in~\cite{Holighaus2019continuous}, except that in our case the warping is defined in the time domain, see also~\cite{jaillet2009nonstationary} for similar ideas. Given the symmetries of the STFT (Fourier transforming $f$ is essentially equivalent to rotate its STFT by $\pi/2$), the adaptation of the time-focused transform to frequency focus can be easily done.

The following result\footnote{\bruno{Reference~\cite{warion2024class} actually assumed $h$ continuous and compactly supported, and $\sigmin=1$. The result is extended here to windows $h$ satisfying the decay condition~\eqref{fo:window.assumptions} and $\sigmin >0$.}} was proved in~\cite{warion2024class}
\begin{theorem}\label{theo_bounding_time_marseillan_transform}
	Let $h\in C^1(\R)$ satisfy the decay hypotheses~\eqref{fo:window.assumptions}. 
    Let $f\in L^2(\R)$ and $\sigma_f\in \bruno{\Phi_{\sigmin}}$ be an associated time focus function, then
	\begin{equation}
    \label{fo:continuous.frame.like}
		c_f \|f\|_{2}^2 \leqslant\| M f\|_{L^2(\R^2)}^2 \leqslant C_f \|f\|_{2}^2\ ,
	\end{equation} 
	where
	\begin{equation}
		c_f = \bruno{\sigmin}\,\inf_{t\in\R}\left\{ \int_\R|h(x\sigma_f(x+t))|^2 dx \right\}>0 \ ,
	\end{equation}
	and
	\begin{equation}\label{eq_borne_sup_norme_Mf}
		C_f = \int_\R |h(\sigma_f(t)t)|^2\sigma_f(t)dt < \infty \ .
	\end{equation}
\end{theorem}

\begin{remark}
\begin{itemize}
\item 
Inequality~\eqref{fo:continuous.frame.like} is reminiscent of a continuous frame condition~\cite{Ali1993continuous,Kaiser1994friendly}, however it is not, since $c_f$ and $C_f$ depend on the analyzed function $f$. However, they depend on $f$ only through $\sigma_f$, therefore the mapping $f\to\sigma_f$ plays a key role. Note that if $\sigma_f$ is constant, we recover the continuous frame condition.
\item 
For the same reason, the strictly positive lower bound doesn't yield injectivity, and therefore invertibility for $M$.
\end{itemize}
\end{remark}

\subsection{Linear transform and remarks}

Note that if the focus function $\sigma\in \bruno{\Phi_{\sigmin}}$ is independent of $f$, then the obtained transform
\begin{equation}
    \forall t,\omega\in\R\ ,\quad \L_\sigma f(t,\omega) = \langle f, h_{t,\omega,\sigma}\rangle
\end{equation}
is not only linear but also bounded from $L^2(\R)$ into $L^2(\R^2)$ and injective, as a consequence of Theorem~\ref{theo_bounding_time_marseillan_transform}. We are then in a continuous frame situation, and $\L_\sigma$ may be given a left inverse. For example, the linear operator
\begin{equation}
\L^\dag_\sigma \!=\! \frac{1}{k_\sigma} \L_\sigma^\ast\, ,\quad \text{where}\quad 
k_\sigma(x) \!=\!\! \int_\R\!\! \sigma(t)\,h(\sigma(t)(x-t))^2 dt\ 
\end{equation}
and $\L_\sigma^\ast$ is the adjoint of $\L_\sigma$, is a left inverse of $\L_\sigma$ (by Theorem~\ref{theo_bounding_time_marseillan_transform}, $k_\sigma$ is bounded from below, away from 0, see~\cite{warion2024class}). By a slight abuse of notation we have denoted here by $1/k_\sigma$ the operator of pointwise multiplication by $1/k_\sigma$. 
%

\begin{lemma}
\label{le:Fp}
Under the the above assumptions, for all $p\ge 2$,
\begin{equation}
\sup_{t\in\R}\|\L_\sigma f((t,\cdot)\|_{p}<\infty\ .
\end{equation}
\end{lemma}
\begin{proof}[Sketch of the proof]
The assumptions imply $f\in L^2(\R)$ and $h\in L^2(\R)\cap L^\infty(\R)$, then the  Cauchy-Schwarz inequality gives $|\L_\sigma f(t,\omega)|\le \sqrt{\sigma_f(t)}\,\|f\|_2\,\|h\|_2\le \|\sigma\|_\infty^{1/2}\|f\|_2\,\|h\|_2$. Also, $\|\L_\sigma f(t,\cdot)\|_2\le \|f\|_2\,\|h\|_\infty$. The bound for general $p$ results from interpolation, and gives $\|\L_\sigma f(\cdot,t)\|_p^p \le \|\sigma\|_\infty^{(p-2)/2} \|h\|_\infty^2\|h\|_2^{p-2}\|f\|^p_2$, for all $t\in\R$.
\end{proof}

\section{Continuity and focus}
\label{sec:continuity.focus}

\subsection{Continuity results}

We first provide an estimate for the $L^2$ norm of difference of focused transforms with different focus functions.
\begin{prop}[$L^2$ norm error]\label{prop_L2_norm_differrence_linear_op}
    Let $\sigma,\kappa \in \bruno{\Phi_{\sigmin}}$ and $h\in C^1$ be a window function satisfying the assumptions~\eqref{fo:window.assumptions}. Then for $f\in L^2$ we have the following $L^2$ norm difference,
    \begin{equation}
        \|\L_\sigma f - \L_\kappa f \|_{L^2(\R^2)} \leqslant C_1(h,\sigma,\kappa)\,\|\sigma-\kappa\|_{\infty} \|f\|_{2},
    \end{equation}
    where 
    \begin{equation*}
        C_1(h,\sigma,\kappa):=\tfrac{\sqrt{2}}{2}\frac{\|\psi_h\|_{\infty}}{\sigmin} + \sqrt{\tfrac{\pi}{2}}\min\{\|\kappa\|_{\infty}^{1/2},\|\sigma\|_{\infty}^{1/2}\}\frac{\|\vphi_h\|_{\infty}}{\sigmin^{3/2}}.
    \end{equation*}
\end{prop}
\begin{proof}[Sketch of the proof]
Define $h_\sigma(x,t) = h(\sigma(t)(x-t))$.
An explicit calculation (using Fubini's lemma and change of variables) gives
\[
\| \L_\sigma f - \L_\kappa f \|^2_{L^2(\R^2)} = \iint_{\R^2} |f(x)|^2 \,|\Sigma_{\sigma,\kappa}(x,t)|^2\, dt dx\ ,
\]
where
\[
\Sigma_{\sigma,\kappa}(x,t) = \sqrt{\sigma(t)}h_\sigma(x,t) -\sqrt{\kappa(t)}h_\kappa(x,t)\ .
\]
Combining the decay assumptions~\eqref{fo:window.assumptions} on the window with Lagrange's mean value theorem yields, after a few manipulations, the desired upper bound.
\end{proof}
\begin{remark}
Notice that the constant $C_1$ depends on $\sigma$ and $\kappa$ only through their $L^\infty$ norm. We discuss in Section~\ref{sec:continuity.focus} a construction of entropy-based focus functions whose $L^\infty$ norm can be controlled.
\end{remark}
Such a Proposition expresses a form of continuity with respect to the focus function. Having such norm control, we can now answer the first question: If we consider an estimate $\sigma_\eps$ of the focus function of a signal $f$, how far are we from the proper transform $\L_{\sigma_\eps}f$? That error may be estimated by the following consequence of Proposition~\ref{prop_L2_norm_differrence_linear_op}. 
\begin{corollary}\label{prop_linearisation_error}
With the same assumptions as before, let $f\in L^2$, $\varepsilon>0$, and let $\sigma_\varepsilon\in \bruno{\Phi_{\sigmin}}$ be such that $\|\sigma_f-\sigma_\varepsilon\|_\infty \le \varepsilon$. Then
\begin{equation}
\|\L_{\sigma_\eps} f - Mf\|_{L^2(\R^2)} \leqslant \eps\, C_1'(h,\sigma_f)\,\|f\|_{2}\ ,
\end{equation}
where
\[
C_1'(h,\sigma_f) =  \tfrac{\sqrt{2}}{2}\,\frac{\|\psi_h\|_{\infty}}{\sigmin} + \sqrt{\tfrac{\pi}2}\,\frac{\|\varphi_h\|_{\infty}}{\sigmin^{3/2}} \,\|\sigma_f\|_{\infty}^{1/2}\ .
\]
\end{corollary}
\noindent Similarly, we have the $L^2$ norm difference for the left inverse
\begin{prop}
\label{prop_L2_norm_difference_pinv}
Let $\sigma,\kappa\in\bruno{\Phi_{\sigmin}}$, $F\in L^2(\R^2)$ then
\begin{equation}
\|\L_\sigma^\dag F - \L_\kappa^\dag F\|_{2} \leqslant C_2\|\sigma -\kappa\|_{\infty}\|F\|_{L^2(\R^2)}\ ,
\end{equation}
for some constant $C_2$ depending on $h,\sigma,\kappa,\sigmin$. 

\end{prop}
\noindent The proof essentially follows the same lines as the previous one, the result is quite similar.

For the next section, we need another norm error control but for the $L^\infty_t L^p_\omega$, $p\in[2,+\infty)$.
\begin{prop}
\label{prop_Lp_norm_difference_linear_op}
Let $h\in C^1(\R)$ satisfy conditions~\eqref{fo:window.assumptions}, $\sigma,\kappa\in \bruno{\Phi_{\sigmin}}$, $p\in[2,+\infty)$ and $f\in L^2$. Then for all $t\in\R$,
\begin{equation}
\| \L_\sigma f(t,\cdot) - \L_\kappa f(t,\cdot) \|_{p}\leqslant K_p(h,\kappa,\sigma)\, \|\sigma - \kappa\|_{\infty}\,\|f\|_{2} \ ,
\end{equation}
with $K_p=K_p(h,\kappa,\sigma)$ defined as
	\begin{equation*}\label{fo:Lp_norm_difference_linear_op}
			K_p=\frac{1}{2\sigmin}\|\sigma\|_{\infty}^{\tfrac{p-2}{2p}}\|h\|_\infty^{2/p}\|h\|_2^{\tfrac{p-2}{p}}\!+\frac{\pi^{1/2}}{2^{\tfrac{p+4}{2p}}\sigmin^{\tfrac{3p+4}{2p}}}\|\kappa\|_{\infty}^{1/2}\|\vphi_h\|_\infty  \, .
		\end{equation*}
\end{prop}
Interestingly enough, the upper bound in Proposition~\ref{prop_Lp_norm_difference_linear_op} depends only on $p$ through the constants in $K(\kappa,\sigma,h)$.

\begin{proof}[Sketch of the proof]
The ingredients are mainly the same as for Proposition~\ref{prop_L2_norm_differrence_linear_op}. Write $\Sigma_{\sigma,\kappa}=\sqrt{\sigma}(h_\sigma\!-\!h_\kappa)\! +\! (\sqrt{\kappa}\!-\!\sqrt{\sigma})h_\kappa$ in the l.h.s. of~\eqref{fo:Lp_norm_difference_linear_op} and use the triangle inequality. The two resulting terms can be bounded using the decay assumptions~\eqref{fo:window.assumptions} on $h$ together with Lagrange's mean value theorem for the first term and Lemma~\ref{le:Fp} for the second one.
\end{proof}
\smallskip

\subsection{Entropy-based focus functions}

In~\cite{warion2024class}, numerical illustrations for focused transforms were provided, which showed the practical effectiveness of the proposed transforms. In particular, it was shown that it is relevant for the focus function to depend nonlinearly on the function being analyzed, and to limit the impact of its (local) amplitude. Focus functions based on norm quotients, such as entropies (also considered in related, but finite-dimensional problems in~\cite{Liuni2013automatic} for example), appeared to be good candidates. Entropies were used to measure the spread of functions in mathematical contexts~\cite{beckner1975inequalities,BialynickiBirula2006formulation} and applied contexts~\cite{Coifman1992entropy,Ricaud2013refined}. Of course, numerical illustrations are intrinsically finite-dimensional, and so do not raise the same problems as the situations discussed here. In this section, we propose examples of focus functions adapted to continuous-time situations.

We recall the definition of (differential) entropies: given a non-negative $L^1$ function $\rho$ such that $\int_\R \rho(x)\,dx = 1$, the Rényi entropies $R_p$ and Shannon entropy $H$ are defined by
\begin{equation}
R_p[\rho] = \frac{1}{1\!-\!p}\ln \|\rho\|_{p}^{p}\ ,\quad
H[\rho] = - \int_\R \rho(x)\ln\rho(x)\,dx\ .
\end{equation}
It may be shown that $H[\rho] = \lim_{p\to 1} R_p[\rho]$.
Unfortunately, the entropies defined in the continuous setting lack some of the nice properties of discrete entropies. In particular, they can take negative values and even be $-\infty$ in some cases. We avoid such situations by suitable regularization and investigate below some properties that are relevant for the current work.

Given an $L^1\cap L^p$ function $v$, a suitable function $\rho$ to build $R_{p}[\rho]$ can be defined as $\rho=|v|/\|v\|_1$. Revisiting the finite-dimensional construction of~\cite{warion2024class}, we introduce modified versions. Let $p>2$, $\kappa\in \bruno{\Phi_{\sigmin}}$ be a reference focus function, and $f\in L^2(\R)$, we set
\begin{equation}
\label{fo:pdf}
\rho_{f,\kappa}^r(t,\omega) = \frac{\left|\mathcal{L}_\kappa f(t,\omega)\right|^2 + r\|f\|_2^2\, u(\omega)}{\left\|\mathcal{L}_\kappa f(t,\cdot)\right\|^2_{2} + r\|f\|_2^2}
\end{equation}
where $r>0$ is a regularization parameter, and $u$ is a smooth non-negative function such that $\int_\R u(\omega)d\omega = 1$. 

This ensures that the denominator is bounded from below by $r\|f\|_2^2>0$ (and from above, thanks to Lemma~\ref{le:Fp}). The presence of the function $u$ in~\eqref{fo:pdf} ensures that the numerator is also bounded from below away from 0. Examples of choices for $u$ include gaussian functions $u(\omega)=e^{-\omega^2/a^2}/a\sqrt{\pi}$, with $a>0$ a scale parameter. This choice ensures that $\rho_{f,\kappa}^r(t,\omega)$ is independent of the overall amplitude of $f$, which is a desirable property, as stressed in~\cite{warion2024class}.
%

%
%
From this, for every $t\in\R$, the Rényi entropy 
\begin{equation}
g_{f,\kappa,p}(t)=R_p[\rho_{f,\kappa}^r(t,\cdot)]
\end{equation}
can be computed (we omit the explicit dependence on $r$ for simplicity). Bounds are provided by the following result.
\begin{lemma}
\label{le:entropy.bds}
\begin{enumerate}
\item 
$g_{f,\kappa,p}$ is continuous on the real axis, and  
\[
\lim_{t\to\pm\infty} g_{f,\kappa,p}(t)= \frac{p}{1-p}\ln \|u\|_{p}\ .
\]
\item 
$g_{f,\kappa,p}$ satisfies the bounds
\begin{eqnarray*}
g_{f,\kappa,p}(t)&\le&  \frac{p}{1-p}
\ln\left(\!\frac{r\|f\|_2^2\,\|u\|_{p}}{\left\|\mathcal{L}_\kappa f(t,\cdot)\right\|^2_{2} + r\|f\|_2^2}\!\right)\\
g_{f,\kappa,p}(t)&\ge&  \frac{p}{1-p}\ln\left(\!\frac{\pierre{\left\|\mathcal{L}_\kappa f(t,\cdot)\right\|_{2p}^2} + r\|f\|_2^2\,\|u\|_{p}}{r\|f\|_2^2}\!\right)\ .
\end{eqnarray*}
\end{enumerate}
\end{lemma}
\begin{proof}[Sketch of the proof]
\begin{enumerate}
\item
By general arguments, the map $t\to \L_\sigma(t,\omega)$ is continuous and tends to 0 at $\pm\infty$. This implies the continuity and the limits at $\pm\infty$ of $g_{f,\kappa,p}$.
\item
The bounds follow from direct calculations.
\end{enumerate}

\vspace{-4mm}
\end{proof}
These bounds depend on $t$. Uniform bounds can then be obtained using Lemma~\ref{le:Fp}
\begin{corollary}
$g_{f,\kappa,p}$ satisfies the uniform bounds
\begin{eqnarray*}
g_{f,\kappa,p}(t)\!\!&\!\!\le\!\!&\!\!  \frac{p}{1\!-\!p}
\ln\left[\frac{\|u\|_p}{1+\frac{1}{r}\|\kappa\|_\infty \|h\|_2^2}\right]\\
g_{f,\kappa,p}(t)\!\!&\!\!\ge\!\!&\!\!  \frac{p}{1\!-\!p}\ln\left[\|u\|_p +\pierre{ \frac{1}{r}\|\kappa\|_\infty^{1-1/p} \|h\|_\infty^{2/p}\|h\|_2^{2-2/p}}\right]\, .
\end{eqnarray*}
\end{corollary}
\bruno{
This suggests to introduce the following candidate focus functions: for $f\in L^2(\R)$,
\begin{equation}
\label{fo:entropic.focus}
\sigma_{f,\kappa,p}(t) = 1 + A \left[g_{f,\kappa,p}(t)- \frac{p}{1-p}\ln\left(\|u\|_p\right) \right]\ ,
\end{equation}
for some constant amplitude $A>0$.
}
Actually, for $f\to \sigma_{f,\kappa,p}$ to define a focus function, the regularization constant $r$ must then be large enough to ensure strict positivity:
\bruno{
\begin{prop}
\label{prop:rmin}
$\sigma_{f,\kappa,p}\in\Phi_{\sigmin}$ if
\[
r \ge r_{\min}(\kappa)=\frac{A\,p}{(p-1)(1-\sigmin)}\ \frac{\|\kappa\|_\infty^{1-1/p} \|h\|_\infty^{2/p}\|h\|_2^{2-2/p}}{\|u\|_p}\ ,
\]
and satisfies, for all $t\in\R$ 
\[
\sigmin\le \sigma_{f,\kappa,p}(t) \le  1 + \frac{A\,p}{p\!-\!1}
\ln\left[1+\frac{1}{r}\|\kappa\|_\infty\,\|h\|_2^2\right]\ .
\]
\end{prop}
}
Notice that the presence of $1-\sigma_{\min}$ in the denominator of $r_{\min}$ imposes $\sigma_{\min}>1$.

\begin{remark}
\begin{itemize}
\item 
\bruno{The above result suggests that $r$ can depend on $\kappa$ through $\|\kappa\|_\infty$. From now on, for the sake of simplicity, we will limit ourselves to the minimal value $r=r_{min}(\kappa)$ given in Proposition~\ref{prop:rmin}}.
\item 
The focus function used in the definition of the transform involves a constant reference focus function $\kappa=1$, in which case we write $\sigma_{f,p} = \sigma_{f,1,p}$ for simplicity. The general case is used below to get stability estimates.
\end{itemize}
\end{remark}
We now analyze the continuity of $\sigma_{f,\kappa,p}$ with respect to the reference focus function $\kappa$, in the case the resularization parameter $r$ is set to its minimal possible value $r=r_{\min}(\kappa)$.
\begin{theorem}
Let $p\ge2$, $\kappa_1,\kappa_2\in \bruno{\Phi_{\sigmin}}$, and let $h\in C^1(\R)$ satisfy the decay condition~\eqref{fo:window.assumptions}. If $r_1=r_{\min}(\kappa_1)$ and $r_2=r_{\min}(\kappa_2)$ then for any $f\in L^2(\R)$,
\begin{equation}
\left\|\sigma_{f,\kappa_2,p} - \sigma_{f,\kappa_1,p}\right\|_\infty\leqslant C_3\, \frac{Ap}{p-1}\, \|\kappa_2-\kappa_1\|_\infty\ ,
\end{equation}
for some constant
\[
C_3 = C_3(p,\|\kappa_1\|_\infty,\|\kappa_2\|_\infty,h,\|u\|_p).
\]
\end{theorem}
\begin{proof}[Sketch of the proof]
Set $Q_i(t,\omega) = |\L_{\kappa_i}f(t,\omega)|^2\!+\!r_i\|f\|_2^2 u(\omega)$ for $i=1,2$. We have that
\[
\left|\ln\left(\frac{\|Q_1(t,\cdot)\|_p^p}{\|Q_1(t,\cdot)\|_1^p}\right) - 
\ln\left(\frac{\|Q_2(t,\cdot)\|_p^p}{\|Q_2(t,\cdot)\|_1^p}\right)\right| \le p (V_p(t)+V_1(t))\ ,
\]
with 
\[
V_p(t) = \left|\ln\left(\frac{Q_1(t,\cdot)\|_p}{Q_2(t,\cdot)\|_p}\right)\right|\le
\left|\frac{\|Q_1(t,\cdot)\|_p-\|Q_2(t,\cdot)\|_p}{r_2\|f\|_2^2\,\|u\|_p}\right|\ .
\]
'
The first factor can be bounded using the triangle inequality and Lemma~\ref{le:Fp}, the second being bounded using Proposition~\ref{prop_Lp_norm_difference_linear_op}. Also using the definition of $r_1,r_2$ and some inequalities gives us that for some constant $C(\sigmin,h,u)$ we have $|r_1-r_2|\le C \|\kappa_1 - \kappa_2\|_\infty$.

Similar arguments give 
\begin{equation*}
    V_1(t) \le \frac{|\|Q_1(t,.)\|_1 - \|Q_2(t,.)\|_1|}{r_2\|f\|_2^2}\ ,
\end{equation*}

\vspace{-3mm}
\noindent and 
\begin{align*}
    |\|Q_1(t,.)\|_1 - \|Q_2(t,.)\|_1| & \le \||\L_{\kappa_1}f(t,.)|^2-|\L_{\kappa_2}f(t,.)|^2\|_1 \\
    & \hphantom{aaaa} + |r_1-r_2|\,\|f\|_2^2\ .
\end{align*}
The first term can be controlled by
\begin{align*}
  &\||\L_{\kappa_1}f(t,.)|^2-|\L_{\kappa_2}f(t,.)|^2\|_1  \\
 &\hphantom{aaaa}\le  \|\L_{\kappa_1}f(t,\cdot)\!+\!\L_{\kappa_2}f(t,\cdot)\|_2\, \|\L_{\kappa_1}f(t,\cdot)\! -\! \L_{\kappa_2}f(t,\cdot)\|_2\, ,
\end{align*}
and the second term is controlled as before, leading to a corresponding upper bound.

Putting together the two bounds yields the result.
\end{proof}

Using similar computations we also obtain the continuity of the focus function along the analysed signal.
\begin{prop}
    Let $f,g\in L^2$ and $\kappa\in \bruno{\Phi_{\sigmin}}$, $p\in[2,+\infty)$ then
    \begin{equation}
        \|\sigma_{f,\kappa,p} - \sigma_{g,\kappa,p}\|_\infty \le C_4\,\frac{Ap}{p-1}\,\|f-g\|_2, 
    \end{equation}
    where
    \begin{align*}
        C_4 :&=  \frac{\|f\|_2+\|g\|_2}{r \max\{\|f\|_2^2,\|g\|_2^2\}} \\
        & \times\left( \|\kappa\|_\infty^{\tfrac{p-1}{p}}\|h\|_\infty^{\tfrac{2}{p}} \|h\|_2^{\tfrac{2(p-1)}{p}} +  \|h\|_\infty^2+r \right)
    \end{align*}
\end{prop}
This shows that if $f\rightarrow g$ in $L^2$ then $\sigma_f \rightarrow \sigma_g$ in $L^\infty$, and therefore comforts us in the choice of entropy-based focus functions. We stress that $C_4$ depends on $f,g$ and is then not a constant; but this doesn't affect the continuity result. 

\section{Summary and perspectives}
\label{se:conclu}
We have derived in this contribution stability estimates for the time-focused time-frequency transform introduced in~\cite{warion2024class}, and investigated the properties of regularized entropy-based focus functions, which turn out to be good candidates for defining the time-focused transforms (as already argued in~\cite{warion2024class}, on the basis of numerical illustrations). While other constructions of focus functions could be considered, we believe that the ones we have studied here represent interesting first examples, both from a theoretical and an application point of view. 

In this paper we could only sketch the proofs of the main results; more details will be published in a forthcoming paper, together with corresponding results on the frequency-focused transform introduced in~\cite{warion2024class}. Also, since the present paper is focused on the continuous setting, we did not address here discretization issues. The results of~\cite{Holighaus2019continuous} provide a natural setting to address the discrete case, to which regularized entropic focus functions can be adapted.

As alluded to in the introduction,  an invertibility result for the focused transform is still missing. Suppose that the original focus function $\sigma_f$ is computed from a standard STFT with a fixed size window, say $\L_1f$. Since fixed-focus versions $\L_\kappa$ are left invertible, an approximate inverse for $M$ can be obtained (and controlled thanks to Proposition~\ref{prop_L2_norm_difference_pinv}) provided that an estimate for $\sigma_f$ is available. Such an estimate could be provided by a variant of the entropy-based focus function computed from $Mf$ instead of $\L_1 f$ (replacing $\L f$ with $Mf$ in the probability density function~\eqref{fo:pdf}). The process could then be iterated. The study of such approaches is ongoing work.

Notice finally that we have limited ourselves to $L^2$ functions which, together with assumptions on the window and the focus functions, are enough to guarantee that $\sup_t \|\L_\sigma(t,\cdot)\|_p<\infty$ for example. It would be interesting to extend the results presented here to larger classes of analyzed functions or distributions, probably with more restrictive assumptions on the window $h$.

\newpage
\bibliographystyle{plain}  
\bibliography{biblio.bib}
\newpage

\end{document}